\documentclass[12pt,a4paper]{article}
\usepackage[english]{babel}
\usepackage{amsmath,amssymb,amsthm}

\def\mapright#1{\smash{
   \mathop{\longrightarrow}\limits^{#1}}}

\def\mapdown#1{\Big\downarrow
   \rlap{$\vcenter{\hbox{$\scriptstyle#1$}}$}}

\newtheorem{theorem}{Theorem}[section]
\newtheorem{lemma}[theorem]{Lemma}
\newtheorem{proposition}[theorem]{Proposition}
\newtheorem{corollary}[theorem]{Corollary}
\newtheorem*{definition}{Definition}
\newtheorem{example}[theorem]{Example}
\newtheorem{question}[theorem]{Question}
\newtheorem{remark}[theorem]{Remark}

\def\P{\mathbb P}

\def\O{{\mathcal O}}
\def\J{{\mathcal J}}

\def\A{{\mathcal A}}
\def\B{{\mathcal B}}
\def\Q{{\mathcal Q}}
\def\U{{\mathcal U}}
\def\Hom{{\rm Hom}}

\begin{document}

\title{Schwarzenberger bundles of arbitrary rank on the projective
space}
\author{Enrique Arrondo}
\date{}
\maketitle

\begin{abstract}
We introduce a generalized notion of Schwarzenberger bundle on the
projective space. Associated to this more general definition, we give
an ad-hoc notion of jumping subspaces of a Steiner bundle on
$\P^n$ (which in rank $n$ coincides with
the notion of unstable hyperplane introduced by Vall\`es, Ancona and
Ottaviani). For the set of jumping hyperplanes, we find a sharp
bound for its dimension. We also classify those Steiner bundles
whose set of jumping hyperplanes have maximal dimension and
prove that they are generalized Schwarzenberger bundles.

\end{abstract}


\leftline{2000 Mathematics subject classification: 14F05, 14N05}

\section*{Introduction}

In \cite{Sch}, Schwarzenberger constructed some particular
vector bundles $F$ of rank $n$ in the projective space
$\P^n$, related to the secant spaces to rational normal curves and
having a resolution of the form
$$0\to\O_{\P^n}(-1)^{\oplus s}\to\O_{\P^n}^{\oplus t}\to
F\to0.$$
Arbitrary vector bundles on $\P^n$ admitting such a
resolution and having arbitrary rank (necessarily at least $n$)
has been widely studied since then. These general bundles were
called Steiner bundles by Dolgachev and Kapranov in \cite{DK}, because
of their relation with the classical Steiner construction of
rational normal curves. In that paper, the authors relate some
Steiner bundles of rank $n$ (the so called logarithmic bundles) to
configurations of hyperplanes in
$\P^n$. In fact, to a general configuration of $k$ hyperplanes
they assign a Steiner bundle and, if this is not a
Schwarzenberger bundle, there is a Torelli-type result in the
sense that the configuration of hyperplanes can be
reconstructed from the bundle (this is proved in \cite{DK} only for
$k\ge 2n+3$, and in general by Vall\`es in \cite{V1}).

The result of Vall\`es and other related results by him and
Ancona and Ottaviani (see \cite{AO}) are based on considering special
hyperplanes associated to Steiner bundles of rank $n$, the
so-called unstable hyperplanes. In particular, they
prove that a Steiner bundle of rank $n$ is one of those
constructed by Dolgachev and Kapranov if and only if it
possesses at least $t+1$ unstable hyperplanes (\cite{AO} Corollary
5.4) and if it has at least $t+2$ unstable hyperplanes then it
is a Schwarzenberger bundle and the set of unstable hyperplanes
forms a rational normal curve (\cite{V1} Th\'eor\`eme 3.1). Hence,
except in the last case, one recovers the original configuration
of hyperplanes from its corresponding Steiner bundle. On the
other hand, it is also true that, starting from a rational normal
curve instead of a finite number of hyperplanes and
constructing its corresponding Schwarzenberger bundle, one can
still reconstruct the rational normal curve from the set of
unstable hyperplanes.

The starting point of this paper is the last of the above
results, i.e. the correspondence between Schwarzenberger bundles
and rational normal curves. First we introduce a generalized
notion of Schwarzenberger bundle, which will be a Steiner bundle
(of rank arbitrarily large) obtained from a triplet $(X,L,M)$,
where $X$ is any projective variety and $L,M$ are globally
generated vector bundles on $X$ of respective ranks $a,b$. In this
context, the original vector bundles constructed by Schwarzenberger
are those obtained from triplets in which $X=\P^1$ and $L,M$ are line
bundles on $\P^1$. Independently, Vall\`es in \cite{V2} has recently
given a similar definition in the case $a=b=1$, assuming that $X$ is a
curve, $M$ is very ample and $H^1(L\otimes M^{-1})=0$, but he allows
$L$ to be just a coherent sheaf (so that $F$ is just a coherent
sheaf, not necessarily locally free). He also generalizes the notion
of logarithmic bundles to arbitrary rank and extends the Torelli-type
results for configurations of lines in
$\P^2$. 

The first main problem we want to study is the following:

\begin{question}\label{question}
When is a Steiner bundle a generalized Schwarzenberger bundle? 
\end{question}

In order to answer this question, one needs to see whether it is
possible to associate a triplet $(X,L,M)$ to a given Steiner
bundle. Following the main ideas in \cite{DK}, \cite{AO} and \cite{V1}, we
observe that, for Schwarzenberger bundles, any point of $X$
yields a special subspace of $\P^n$, which we call
$(a,b)$-jumping subspace (in fact we will introduce the more natural
notion of jumping pair). This notion generalizes the notion of
unstable hyperplane in \cite{AO} and \cite{V1}, so that we naturally
wonder about the following Torelli-type problem:

\begin{question}\label{otrapregunta}
For which triplets $(X,L,M)$ does it happen that all the jumping
subspaces come from points of $X$?
\end{question}

In this paper, we give a positive answer to Questions
\ref{question} and \ref{otrapregunta} when $a=b=1$ and the set of
jumping subspaces (which in this case are hyperplanes), or more
generally the set of jumping pairs, has maximal dimension. More
precisely, when $a=b=1$ we first provide a sharp bound for the
dimension of the set of jumping pairs of Steiner bundles. Then we
classify all Steiner bundles for which the set of jumping pairs has
maximal dimension, showing that in all cases they are generalized
Schwarzenberger bundles and that the variety $X$ in the triplet is
obtained from the set of jumping pairs.

I want to stress the fact that, despite of the apparently abstract
notions developed in the paper, most of the inspiration and
techniques come from classical projective geometry (varieties of
minimal degree, Segre varieties, linear projections,...).  

The paper is structured as follows. In a first section, we
recall the main properties of Steiner bundles and introduce our
generalized notion of Schwarzenberger bundle. We present four
examples of Schwarzenberger bundles and prove (Proposition
\ref{unicoschwarzenberger}) that, in rank $n$, our definition coincides
with the original Schwarzenberger bundles.

In a second section, we introduce the notion of $(a,b)$-jumping
subspaces and pairs of a Steiner bundle. In the particular case
$a=b=1$, we show (Theorem \ref{teoremon}) that the set of jumping
pairs has dimension at most $t-n-s+1$ and that, if $n=1$ or $s=2$,
any Steiner bundle is a Schwarzenberger bundle (thus generalizing to
our general context the known result for rank $n$).

Finally, in the third section we classify Steiner vector bundles
whose set of jumping pairs has maximal dimension (Theorem
\ref{megateoremon}), showing that, in this case, they are Schwarzenberger
bundles, precisely the examples introduced in the first section. We
include, as a first application of our theory, an improvement
(Corollary \ref{re}) for line bundles of a result of Re (see \cite{R})
about the multiplication map of sections. We finish with some
remarks about the difficulty of the case of arbitrary
$a,b$, and with some possible generalization of our definition to
arbitrary varieties.

This paper has been written in the framework of the research
projects MTM2006-04785 (funded by the Spanish Ministry of
Education) and CCG07-UCM/ESP-3026 (funded by the University
Complutense and the regional government of Madrid). I also
want to thank Sof\'{\i}a Cobo, whose remarks after a careful reading
of a preliminary version helped a lot to improve the
presentation of the paper and suggested the current improvement of
Theorem \ref{megateoremon} (originally stated for the dimension of
$J(F)$).


\section{Generalized Schwarzenberger bundles}

\noindent{\bf General notation.} We will always work over a fixed
algebraically closed ground field $k$. We will use the notation
that, for a vector space $V$ over $k$, the projective space
$\P(V)$ will be the set of hyperplanes of $V$ or equivalently the
set of lines in the dual vector space $V^*$. If $v$ is a nonzero
vector of $V^*$, we will write $[v]$ for the point of $\P(V)$
represented by the line $<v>$ spanned by $v$. On the other
hand, we will denote by $G(r,V)$ the Grassmann variety of
$r$-dimensional subspaces of a vector space $V$.

\medskip

Recall first the definition of Steiner bundle, in which we
will include for convenience the invariants of the resolution.

\begin{definition} {\rm
We will call {\it $(s,t)$-Steiner bundle} over $\P^n$ to a vector
bundle $F$ with a resolution 
$$0\to S\otimes\O_{\P^n}(-1)\to T\otimes\O_{\P^n}\to
F\to0$$
where $S,T$ are vector spaces over $k$ of respective
dimensions $s$ and $t$ (observe that the rank of $F$ is thus
$t-s$). 
}\end{definition}

\begin{remark}\label{geometria}{\rm
We recall from \cite{DK} the
geometric interpretation of the resolution of a Steiner
bundle. A morphism $\O_{\P^n}(-1)\to T\otimes\O_{\P^n}$ is
equivalent to fixing an $(n+1)$-codimensional linear subspace
$\Lambda\subset\P(T)$ and identifying $\P^n$ with the set,
which we denote by $\P(T)^*_\Lambda$, of hyperplanes of $\P(T)$
containing $\Lambda$. Therefore giving a morphism
$S\otimes\O_{\P^n}(-1)\to T\otimes\O_{\P^n}$ is equivalent to
fixing $s$ linear subspaces
$\Lambda_1,\dots,\Lambda_s\subset\P^{t-1}$ of codimension
$n+1$ with a common parametrization by $\P^n$ of the sets
$\P(T)^*_{\Lambda_i}$  of hyperplanes in $\P^{t-1}$ containing
these $\Lambda_i$. Hence the projectivization of the
fiber of $F$ at any point $p\in\P^n$ is the linear space 
$\P(F_p)\subset\P(T)$ consisting of the intersection of the $s$
hyperplanes of $\P(T)^*_{\Lambda_1},\dots,\P(T)^*_{\Lambda_s}$
corresponding to $p$.
}\end{remark}

We recall in the next lemmas the standard characterization of Steiner
bundles by means of linear algebra, and introduce the notation that we
will use throughout the paper.

\begin{lemma}\label{data}
Given vector spaces $S,T$ over $k$,
the following data are equivalent: 
\begin{enumerate}
\item[(i)] A Steiner bundle $F$ with resolution
$0\to S\otimes\O_{\P^n}(-1)\to T\otimes\O_{\P^n}\to F\to0$.
\item[(ii)] A linear map $\varphi:T^*\to S^*\otimes
H^0(\O_{\P^n}(1))={\rm Hom}(H^0(\O_{\P^n}(1))^*,S^*)$ such
that, for any $u\in H^0(\O_{\P^n}(1))^*$ and any $v\in S^*$,
there exists
$f\in{\rm Hom}(H^0(\O_{\P^n}(1))^*,S^*)$ in the image of
$\varphi$ satisfying $f(u)=v$.
\end{enumerate}
\end{lemma}

\begin{proof} 
Taking duals, giving a morphism $S\otimes\O_{\P^n}(-1)\to
T\otimes\O_{\P^n}$ is equivalent to giving a morphism 
$$\psi:T^*\otimes\O_{\P^n}\to S^*\otimes\O_{\P^n}(1)={\mathcal
H}om(\O_{\P^n}(-1),S^*\otimes\O_{\P^n})$$ 
and this is clearly equivalent to giving linear map
$$\varphi:T^*\to H^0(S^*\otimes\O_{\P^n}(1))=
S^*\otimes H^0(\O_{\P^n}(1))
={\rm Hom}(H^0(\O_{\P^n}(1))^*,S^*).$$ Hence we need to
characterize when the morphism
$\psi$ induced by $\varphi$ is surjective, i.e. when the
fibers of $\psi$ are surjective at any point of
$\P^n$. To this purpose, we observe that, for any point
$[u]\in\P^n$ corresponding to a nonzero vector
$u\in H^0(\O_{\P^n}(1))^*$, the fiber of $\psi$ at $[u]$
is the linear map $T^*\to{\rm Hom}(<u>,S^*)$ consisting of
the restriction of $\varphi$. Hence this map is surjective if
and only if for any $v\in S^*$ there exists $f\in{\rm
Hom}(H^0(\O_{\P^n}(1))^*,S^*)$ in the image of
$\varphi$. This proves the lemma. 
\end{proof}

\begin{lemma} \label{reduced} 
With the notation of Lemma \ref{data}, the
following data are equivalent:
\begin{enumerate}
\item[(i)] A linear subspace $K\subset T^*$ contained in the kernel
of $\varphi$.
\item[(ii)] An epimorphism $F\to K^*\otimes\O_{\P^n}$.
\item[(iii)] A splitting $F=F_K\oplus(K^*\otimes\O_{\P^n})$.
\end{enumerate}
In this case, $F_K$ is the Steiner bundle corresponding, by
Lemma \ref{data}, to the natural map $T^*/K\to S^*\otimes
H^0(\O_{\P^n}(1))$. As a consequence, if $T^*_0$ is the image of
$\varphi$ and $F_0$ is the Steiner bundle corresponding to the
inclusion $T_0^*\to S^*\otimes H^0(\O_{\P^n}(1))$, then
$H^0(F_0^*)=0$ and $F=F_0\oplus(T/T_0)\otimes\O_{\P^n}$. In
particular, $H^0(F^*)=0$ if and only if $\varphi$ is injective. 
\end{lemma}

\begin{proof} 
The equivalence of (ii) and (iii) comes from
the fact that $F$ is generated by its global sections. In the
situation of (i), we have a map $\bar\varphi:T^*/K\to S^*\otimes
H^0(\O_{\P^n}(1))$ which, by Lemma
\ref{data}, induces a Steiner bundle $F_K$. We clearly have a
commutative  diagram
$$\begin{matrix}
&&&&0&&0&\cr
&&&&\downarrow&&\downarrow&\cr
0&\to&S\otimes\O_{\P^n}(-1)&\to&(T^*/K)^*\otimes\O_{\P^n}&\to&F_K&\to&0\cr
&&||&&\downarrow&&\downarrow&\cr
0&\to&S\otimes\O_{\P^n}(-1)&\to&T\otimes\O_{\P^n}&\to&F&\to&0\cr
&&&&\downarrow&&\downarrow&\cr
&&&&K^*\otimes\O_{\P^n}&=&K^*\otimes\O_{\P^n}&&\cr
&&&&\downarrow&&\downarrow&\cr
&&&&0&&0&
\end{matrix}$$
induced by the first two rows, so that the last column yields 
situation (ii). Reciprocally, given an epimorphism
$F\to K^*\otimes\O_{\P^n}$, the resolution of $F$ yields another
epimorphism $T\otimes\O_{\P^n}\to K^*\otimes\O_{\P^n}$, so that we can
consider $K$ as a subspace of $T^*$. We thus get a diagram as above,
now induced by its last two rows. Dualizing the diagram and taking
cohomology, we get that $\varphi:T^*\to S^*\otimes H^0(\O_{\P^n}(1))$
factorizes through $T^*/K$, so that $K$ is contained in the kernel of
$\varphi$, which is situation (i). Observe finally that $F_0$
is nothing but $F_{\ker\varphi}$.
\end{proof}

\begin{definition}{\rm 
With the above notation, we will say that a Steiner bundle is {\it
reduced} if $\varphi$ is injective, i.e. if $H^0(F^*)=0$. 
The Steiner bundle $F_0$ will be called the {\it reduced
summand} of $F$.
}\end{definition}

\begin{remark} \label{rangonreducido} {\rm
Observe that, since there are not Steiner bundles on $\P^n$ of rank
smaller than $n$ (see for instance \cite{DK} Proposition 3.9), any
Steiner bundle of rank $n$ must coincide with its reduced summand,
and hence it is reduced. Notice also that the only reduced Steiner
bundle with $s=1$ is $T_{\P^n}(-1)$. This is why we will only
consider the cases $s\ge2$.
}\end{remark}

Our generalized notion of Schwarzenberger bundle will come from the
following example, in which we will use a slightly more general
framework.

\begin{example} \label{mainconstruction} {\rm
Let $X$ be a projective variety and consider two coherent sheaves
$L,M$ on $X$, and assume $L$ is locally free. If $h^0(M)=n+1$, we
identify $\P^n$ with $\P(H^0(M)^*)$, the set of lines in
$H^0(M)$. Consider the natural composition
$$H^0(L)\otimes\O_{\P^n}(-1)\to
H^0(L)\otimes H^0(M)\otimes\O_{\P^n}\to
H^0(L\otimes M)\otimes\O_{\P^n}$$
For each nonzero $\sigma\in H^0(M)$, the fiber of the above
composition at the point $[\sigma]\in\P^n$ is
$$H^0(L)\otimes<\sigma>\to H^0(L)\otimes H^0(M)\to
H^0(L\otimes M)$$
and, identifying $H^0(L)\otimes<\sigma>$ with
$H^0(L)$ we get that the composition is injective
since it can be identified with
$H^0(L)\mapright{\cdot\sigma}H^0(L\otimes M)$. We thus have a
Steiner vector bundle $F$ defined as a cokernel
$$0\to H^0(L)\otimes\O_{\P^n}(-1)\to
H^0(L\otimes M)\otimes\O_{\P^n}\to F\to0.$$
Observe that the map $\varphi$ of Lemma \ref{data} is, in this
case, the dual of the multiplication map $H^0(L)\otimes H^0(M)\to
H^0(L\otimes M)$. In particular, $F$ is reduced if and only if
this multiplication map is surjective.
}\end{example}

\begin{definition}{\rm
Let $X$ be a projective variety, and let $L,M$ be globally
generated vector bundles on $X$. We will call {\it Schwarzenberger
bundle} of the triplet $(X,L,M)$ to the Steiner vector bundle
constructed in Example \ref{mainconstruction}.
}\end{definition}

\begin{remark} \label{geometriaSchwarzenberger}{\rm 
Following Remark \ref{geometria}, the geometry of a Schwarzenberger
bundle $F$ when $L$ and $M$ are line bundles is related to the
geometry of the map $\varphi_{L\otimes M}:X\to\P(H^0(L\otimes M))$
defined by $L\otimes M$. Indeed, in this case, $\P^n$ is identified
with the complete linear series $|M|$ of effective divisors on
$X$. For each $D\in|M|$, Example \ref{mainconstruction} shows that
the fiber $F_D$ is the cokernel of the map 
$H^0(L)\to H^0(L\otimes M)$ defined by a section of $M$ vanishing
at $D$. Hence the projectivization $\P(F_D)\subset\P(H^0(L\otimes
M))$ is the linear span of the divisor $D$ regarded as a subset
in $\P(H^0(L\otimes M))$ via $\varphi_{L\otimes M}$. Thus Remark
\ref{geometria} is saying that the set of these linear spans can be
constructed by fixing  linear subspaces
$\Lambda_1,\dots,\Lambda_s\subset\P(H^0(L\otimes M))$,
defining common parametrizations of the $\P(H^0(L\otimes
M))^*_{\Lambda_i}$ and taking the intersection of
corresponding hyperplanes.

Therefore, when considering only Schwarzenberger bundles coming
from line bundles, Question \ref{question} can be stated geometrically
as: {\sl Given $s$ linear subspaces
$\Lambda_1,\dots,\Lambda_s\subset\P(T)$ of codimension $n+1$
such that the $\P(T)^*_{\Lambda_i}$ are parametrized by the
same $\P^n$, do the intersections of the corresponding
hyperplanes describe the span of the divisors of some
complete linear system of a variety?}
}\end{remark}

We give now four representative examples of Schwarzenberger
bundles:

\begin{example} \label{schwarzenberger}{\rm 
When $(X,L,M)=(\P^1,\O_{\P^1}(s-1),\O_{\P^1}(n))$, one obtains an
$(s,s+n)$-Steiner bundle of rank $n$, which is precisely the
vector bundle constructed by Schwarzenberger. If
$s=2$, Remark \ref{geometria} provides for any $(2,n+2)$-Steiner
bundle the classical Steiner construction of the rational
normal curve in $\P^{n+1}$, so that the answer to Question
\ref{question} is positive. However, if $s>2$, a general
$(s,s+n)$-Steiner bundle is not a Schwarzenberger bundle
(see \cite{AO} or \cite{V1}).
}\end{example}

\begin{example} \label{superficieminimal}{\rm 
Let $F=\oplus_{i=1}^{t-s}\O_{\P^1}(a_i)$ with $a_i\ge1$ for
$i=1,\dots,t-s$, and assume $\deg F=a_1+\dots+a_{t-s}=s$. Write
$X=\P(F)$ and let $\O_X(h)$ denote the tautological quotient line
bundle (equivalently, $X$ is a smooth rational normal scroll
$X\subset\P^{t-1}$ of dimension
$t-s$ and degree $s$). If $f$ is the class of a
fiber of the scroll, the positivity of the $a_i$ implies that
$L:=\O_X(h-f)$ is globally generated. Then, if $M=\O_X(f)$, the
Schwarzenberger bundle of $(X,L,M)$ is an $(s,t)$-Steiner
bundle on $\P^1$. By the geometric interpretation given in Remark
\ref{geometriaSchwarzenberger}, the fiber of this Schwarzenberger bundle
at any point of $\P^1$ is nothing but the corresponding fiber of the
scroll $X$. Therefore, this Schwarzenberger bundle is precisely the
original $F$. This shows that any ample vector bundle on $\P^1$ is a
Schwarzenberger bundle. Observe
that $F$ can also be regarded as the Schwarzenberger bundle of the
triplet $(\P^1,F(-1),\O_{\P^1}(1))$.
}\end{example}

We consider next the symmetric example with respect to the previous
one, by just permuting $L$ and $M$. Observe that, even if this
permutation produces different vector bundles (in fact defined on
different projective spaces), most of our results on Steiner bundles
will keep some symmetry of this type (for example, in Theorem
\ref{teoremon} the roles of $n+1$ and $s$ are symmetric).

\begin{example} \label{rationalnormalscroll}{\rm 
Let $X$ be a smooth rational normal scroll
$X\subset\P^{t-1}$ of dimension $t-n-1$ and degree $n+1$ defined by
$E=\oplus_{i=1}^{t-n-1}\O_{\P^1}(a_i)$ with $a_i\ge1$ for
$i=1,\dots,t-n+1$. Let $h,f$
be denote respectively the the class of a hyperplane and a fiber of
the scroll. Then, if $L=\O_X(f)$ and $M=\O_X(h-f)$, the
Schwarzenberger bundle of $(X,L,M)$ is a $(2,t)$-Steiner bundle. We
will see in Theorem \ref{teoremon}(iv) that in this case any
$(2,t)$-Steiner bundle is obtained in this way (the case $t=n+2$
is exactly the case $s=2$ of Example \ref{schwarzenberger}). As before, $F$
can also be regarded as the Schwarzenberger bundle of the triplet
$(\P^1,\O_{\P^1}(1),E(-1))$.
}\end{example}

\begin{example} \label{Veronese}{\rm 
The Schwarzenberger bundle of the triplet
$(\P^2,\O_{\P^2}(1),\O_{\P^2}(1))$ is a
$(3,6)$-Steiner bundle $F$ of rank three over
$\P^2$. If we identify this last $\P^2$ with the set of conics of
the Veronese surface $V\subset\P^5$, then the projectivization of
the fiber of $F$ at the element of $\P^2$ corresponding to a conic
$C\subset V$ gives the plane of $\P^5$ spanned by $C$. In fact,
it follows $F=S^2(T_{\P^2}(-1))$ (see \cite{B} p. 615), so that
$F_{|L}=\O_L\oplus\O_L(1)\oplus\O_L(2)$ for any line
$L\subset\P^2$. We will see in Remark \ref{pocosjumping} that a
general $(3,6)$-Steiner bundle is not obtained in this way.
}\end{example}

We end this section by reformulating in terms of our
generalized Schwarzenberger bundles the results of Re about
the multiplication map for vector bundles (we will improve his results
in Corollary \ref{re} in the case of rank one). This will imply in
particular that our generalized Schwarzenberger bundles  of rank $n$
are exactly those constructed originally by Schwarzenberger:

\begin{proposition} \label{unicoschwarzenberger} 
Let $F$ be an $(s,t)$-Steiner bundle on $\P^n$ that is the
Schwarzenberger bundle of a triplet $(X,L,M)$, with rk$(L)=a$ and
rk$(M)=b$. Then:
\begin{enumerate} 
\item[(i)] $t\ge bs+a(n+1)-ab$.
\item[(ii)] If equality holds in (i), then $F$ is the
Schwarzenberger bundle of a triplet $(\P^1,L,M)$, where
$\deg(L)=s-a$ and $\deg(M)=n+1-b$.
\item[(iii)] Any Schwarzenberger bundle of rank $n$ is as in
Example \ref{schwarzenberger}.
\end{enumerate}
\end{proposition}

\begin{proof} 
By \cite{R} Theorem 1 we have $h^0(L\otimes M)\ge bh^0(L)+ah^0(M)-ab$,
which is inequality (i). Moreover, \cite{R} Theorem 2 says that, when
the above inequality is an equality, then there exists a map
$f:X\to\P^1$ and vector bundles $L',M'$ on
$\P^1$ such that $L=f^*L'$, $H^0(L)=f^*H^0(L')$, $M=f^*M'$ and
$H^0(M)=f^*H^0(M')$. This means that $F$ is also the
Schwarzenberger bundle of the triplet $(\P^1,L',M')$. This proves
(ii), since Riemann-Roch theorem for vector bundles on $\P^1$
implies $s=\deg(L')+a$ and $n+1=\deg(M')+b$.

In order to prove (iii), observe that $F$ has rank $n$ if and only if
$t=h^0(L\otimes M)=h^0(L)+h^0(M)-1$. Since $L$ and $M$ are
globally generated, it follows $h^0(L)\ge a$ and $h^0(M)\ge
b$. Therefore
$$t-bs-a(n+1)+ab=(h^0(L)+h^0(M)-1)-bh^0(L)-ah^0(M)+ab=$$
$$=-(b-1)h^0(L)-(a-1)h^0(M)+ab-1\le$$
$$\le(b-1)a+(a-1)b+ab-1=-(a-1)(b-1)\le0.$$
By (i) we have that all inequalities are equalities and in
particular $a=b=1$, and by (ii) we also have that $F$ is the
Schwarzenberger bundle of a triplet $(\P^1,L,M)$, where $L$ and
$M$ are line bundles on $\P^1$ of respective degrees $s-1$ and
$n$, from which the result follows.
\end{proof}

\section{Jumping subspaces of Steiner bundles}

In order to answer Question \ref{question}, one needs to try to
produce a triplet $(X,L,M)$ from a Steiner bundle $F$. The main
idea to find a candidate for $X$ comes from the fact that, since
$M$ is a globally generated vector bundle of rank $b$, any point
$x\in X$ yields a $b$-codimensional subspace
$H^0(M\otimes\J_x)\subset H^0(M)$ consisting of the
sections of $M$ vanishing at $x$. Thus the points of $X$ give
particular linear subspaces of codimension $b$ in the projective
space $\P^n=\P(H^0(M)^*)$ on which the Schwarzenberger bundle is
defined. Hence our goal is to look for some special property
of these linear subspaces for Schwarzenberger bundles and see
whether, for an arbitrary Steiner bundle, the set of 
subspaces satisfying that property could play the role of
$X$. This is the scope of the following:

\begin{lemma} \label{jumpingSchwarzenberger}
Let $F$ be a Steiner bundle over $\P^n$. Then:
\begin{enumerate}
\item[(i)] For any non-empty linear subspace $\Lambda\subset\P^n$,
there is a canonical commutative diagram
$$\begin{matrix}
&&&S^*\otimes
H^0(\J_\Lambda(1))&\mapright{\cong}&H^1(F^*\otimes\J_\Lambda)&\cr
&&&\downarrow&&\mapdown{\phi}\cr
&T^*&\mapright{\varphi}&S^*\otimes
H^0(\O_{\P^n}(1))&\to&H^1(F^*)&\to&0 
\end{matrix}$$ 
\item[(ii)] If $F$ is the Schwarzenberger bundle of the
triplet $(X,L,M)$ and $\Lambda\subset\P^n$ is the 
subspace corresponding to $H^0(M\otimes\J_x)\subset H^0(M)$ for
some $x\in X$, then there exists an
$a$-dimensional linear subspace $A\subset S^*$ such that
$A\otimes H^0(\J_\Lambda(1))$ is in the kernel of $\phi$.
\end{enumerate}
\end{lemma}

\begin{proof} 
Diagram (i) comes by taking cohomology in the dual of the
resolution of $F$ and its twist by $\J_\Lambda$. For (ii), if $F$
is the Schwarzenberger bundle of the triplet $(X,L,M)$, we have 
$$H^0(\O_{\P^n}(1))=H^0(M)^*,\quad S=H^0(L),\quad
T=H^0(L\otimes M)$$ 
and $\varphi$ is the dual of the multiplication map
$H^0(L)\otimes H^0(M)\to H^0(L\otimes M)$. Moreover, if $\Lambda$
is the linear subspace corresponding to
$H^0(M\otimes\J_x)\subset H^0(M)$, for some $x\in X$, we also have
$H^0(\J_\Lambda(1))=H^0(M_x)^*$. It is clear that $\varphi$ maps
$H^0(L_x\otimes M_x)^*$ isomorphically to $H^0(L_x)^*\otimes
H^0(M_x)^*$. Hence, it follows that
$H^0(L_x)^*\otimes H^0(\J_\Lambda(1))$ is mapped to zero in
$H^1(F^*)$. 
\end{proof}

This suggests the following:

\begin{definition} {\rm
Let $F$ be a Steiner bundle over $\P^n$. An {\it $(a,b)$-jumping
subspace of $F$} is a $b$-codimension subspace
$\Lambda\subset\P^n$ satisfying that, with the identification
given in (1), there exists an
$a$-dimensional linear subspace $A\subset S^*$ such that
$A\otimes H^0(\J_\Lambda(1))$ is in the kernel of the natural
map $H^1(F^*\otimes\J_\Lambda)\to H^1(F^*)$. The pair
$(A,\Lambda)$ will be called {\it $(a,b)$-jumping pair of
$F$}. We will write $J_{a,b}(F)$ and $\tilde J_{a,b}(F)$ to
denote respectively the set of $(a,b)$-jumping subspaces 
and the set of $(a,b)$-jumping pairs of $F$. We will also write
$\Sigma_{a,b}(F)$ to denote the set of subspaces $A\subset S^*$ for
which there exists a $b$-codimensional subspace $\Lambda\subset\P^n$
such that $(A,\Lambda)$ is an $(a,b)$-jumping subspace of $F$. A
$(1,1)$-jumping subspace (resp. pair) will be called simply a
{\it jumping hyperplane} (resp. {\it pair}), and we will just
write $J(F)$ (resp. $\tilde J(F)$) to denote the set of jumping
hyperplanes (resp. pairs) of $F$. Similarly we will write
$\Sigma(F):=\Sigma_{1,1}(F)$.
}\end{definition}

We prove next a series of easy properties of jumping
spaces and pairs:

\begin{lemma} \label{propjumping}
Let $F$ be a Steiner bundle over $\P^n$. The following hold:
\begin{enumerate}
\item[(i)] For any $a,b$, the set of $(a,b)$-jumping pairs of $F$
coincides with the set of $(a,b)$-jumping pairs of its reduced summand
$F_0$. In particular, $J_{a,b}(F)=J_{a,b}(F_0)$ and
$\Sigma_{a,b}(F)=\Sigma_{a,b}(F_0)$
\item[(ii)] If $A\subset S^*$ is a linear
subspace of dimension $a$ and $\Lambda\subset\P^n$ is a subspace
of codimension $b$, then $(A,\Lambda)$ is an $(a,b)$-jumping pair of
$F$ if and only if $A\otimes H^0(\J_\Lambda(1))$ is in the image
$T_0^*$ of
$\varphi:T^*\to S^*\otimes H^0(\O_{\P^n}(1))$. 
\item[(iii)] Any $(a,b)$-jumping pair $(A,\Lambda)$ of $F$ induces,
in a canonical way, a split quotient ${F_0}_{|\Lambda}\to
A^*\otimes H^0(\J_\Lambda(1))^*\otimes\O_\Lambda$.
\item[(iv)] If $b=1$, a hyperplane $H\subset\P^n$ is an
$(a,1)$-jumping subspace if and only if there is a quotient
${F_0}{|H}\to\O_H^{\oplus a}$, i.e.  $h^0(F_{|H}^*)\ge
h^0(F^*)+a$.
\end{enumerate}
\end{lemma}

\begin{proof} 
Part (i) is obvious from the splitting (see Lemma \ref{reduced})
$F=F_0\oplus(T/T_0)\otimes\O_{\P^n}$, so that the maps
$H^1(F^*\otimes\J_\Lambda)\to H^1(F^*)$ and
$H^1(F_0^*\otimes\J_\Lambda)\to H^1(F_0^*)$ are the same for any
subspace $\Lambda$. Part (ii) follows at once from Lemma
\ref{jumpingSchwarzenberger}(i).

To prove (iii), let $(A,\Lambda)$ be a jumping pair of $F$. By (ii),
this means that $A\otimes H^0(\J_\Lambda(1))$ can be regarded as a
subspace of $T_0^*$. On the other hand, recall that
$F_0$ is the Steiner bundle constructed (see Lemma \ref{data})
from the inclusion $T_0^*\to S^*\otimes H^0(\O_{\P^n}(1))$. It is
clear that
${F_0}_{|\Lambda}$ is the Steiner bundle constructed from the
composition 
$$T_0^*\to S^*\otimes H^0(\O_{\P^n}(1))\to S^*\otimes
H^0(\O_\Lambda(1))$$
and, since $A\otimes H^0(\J_\Lambda(1))$ is contained in its kernel, 
Lemma \ref{reduced} gives the wanted split quotient. 

Finally, the ``only if'' part of (iv) is (iii). Reciprocally, assume
that there is a quotient ${F_0}{|H}\to\O_H^{\oplus a}$ for some
hyperplane $H\subset\P^n$, which is equivalent, by the splitting
$F=F_0\oplus(T/T_0)\otimes\O_{\P^n}$, to the inequality
$h^0(F_{|H}^*)\ge h^0(F^*)+a$. From the exact sequence
$$0=H^0(F^*\otimes\J_H)\to H^0(F^*)\to
H^0(F_{|H}^*)\to H^1(F^*\otimes\J_H)\to
H^1(F^*)$$
we get that the kernel of $\phi:H^1(F^*\otimes\J_H)\to H^1(F^*)$ has
dimension at least $a$. This kernel, regarded as a subspace of
$S^*\otimes H^0(\J_H(1))$ (see Lemma
\ref{jumpingSchwarzenberger}(i)), is necessarily of the form $A\otimes
H^0(\J_H(1))$, because $H^0(\J_H(1))$ has dimension one. Therefore,
$(A,H)$ is an $(a,1)$-jumping pair and $H$ is an $(a,1)$-jumping
hyperplane. 
\end{proof}

\begin{remark} \label{jumpingmejor}{\rm 
Since Steiner bundles of rank $n$ are reduced (see Remark
\ref{rangonreducido}), part (iv) of Lemma
\ref{propjumping} says that a jumping hyperplane $H$ is
characterized by the condition $H^0(F^*_{|H})\ne0$. This is why in
\cite{AO} and \cite{V1} use the name ``unstable hyperplane'', although in
our general context we preferred the word ``jumping''. Observe
that part (iii) implies that,  if $\Lambda$ is an $(a,b)$-jumping
subspace of $F$, then
$h^0(F^*_{|\Lambda})\ge h^0(F^*)+ab$. However, the converse is not
true, and the proof of (iv) does not work if $a>1$, since an
$ab$-dimensional kernel of $H^1(F^*\otimes\J_\Lambda)\to H^1(F^*)$ is
not necessarily of the form $A\otimes H^0(\J_\Lambda(1))$. However,
one could characterize $(a,b)$-jumping pairs $(A,\Lambda)$ by the
property that, for any hyperplane $H\supset\Lambda$, the pair $(A,H)$
is an $(a,1)$-jumping pair or, similarly, that for any hyperplane
$H\supset\Lambda$ and any line $A'\subset A$ the pair $(A'H)$ is a
jumping pair. 

The reader should notice however that, when $b=n-1$, our notion of
jumping hyperplane does not coincide with the standard notion of
jumping line of a vector bundle in the projective space, even if $n=2$
(i.e. $b=1$). For instance, the Steiner bundle
$F=S^2(T_{\P^2}(-1))$ of Example \ref{Veronese} is uniform,
and even homogeneous, so that it has no jumping lines (in the standard
sense), while any line $L\subset\P^2$ is a jumping hyperplane (in our
sense) because $F_{|L}$ has always a trivial summand.
}\end{remark}

We can give a geometric construction of the sets of the
$(a,b)$-jumping subspaces and pairs, which endows them with a natural
structure of algebraic sets (when $a=b=1$, this is the natural
generalization of the construction given in \cite{AO} \S3 for Steiner
bundles of rank $n$). This also allows to show that, when these sets
satisfy certain conditions of linear normality, the answer to
Question \ref{question} is positive:

\begin{lemma} \label{jumpingSegre}
Let $F$ be a Steiner bundle over $\P^n$ and let $T_0^*\subset
S^*\otimes H^0(\O_{\P^n}(1))$ be the image of
$\varphi$. Consider the natural generalized Segre embedding 
$$\nu:G(a,S^*)\times G(b,H^0(\O_{\P^n}(1)))\to G(ab,S^*\otimes
H^0(\O_{\P^n}(1)))$$ 
(given by the tensor product of subspaces)
and identify $G(b,H^0(\O_{\P^n}(1)))$ with the
Grassmann variety of subspaces of codimension $b$ in $\P^n$.
Then:
\begin{enumerate}
\item[(i)] The set $\tilde J_{a,b}(F)$ of jumping pairs of $F$ is
the intersection of the image of $\nu$ with the subset
$G(ab,T_0^*)\subset G(ab,S^*\otimes H^0(\O_{\P^n}(1)))$.
\item[(ii)] If $\pi_1,\pi_2$ are the respective projections from
$\tilde J_{a,b}(F)$ to $G(a,S^*)$ and
$G(b,H^0(\O_{\P^n}(1)))$, then $\Sigma_{a,b}(F)=\pi_1(\tilde
J_{a,b}(F))$ and $J_{a,b}(F)=\pi_2(\tilde J_{a,b}(F))$.
\item[(iii)] Let $\A,\B,\Q$ be the universal quotient bundles of
respective ranks
$a,b,ab$ of $G(a,S^*)$, $G(b,H^0(\O_{\P^n}(1)))$ and $G(ab,T_0^*)$.
Assume that the natural maps 
$$\alpha:H^0(G(a,S^*),\A)\to H^0(\tilde J_{a,b}(F),\pi_1^*\A)$$
$$\beta:H^0(G(b,H^0(\O_{\P^n}(1))),\B)\to H^0(\tilde
J_{a,b}(F),\pi_2^*\B)$$
$$\gamma:H^0(G(ab,{T'_0}^*),\Q)\to H^0(\tilde J_{a,b}(F),\Q_{|\tilde
J_{a,b}(F)})$$ 
are isomorphisms. Then the reduced summand $F_0$ of $F$ is the
Schwarzenberger bundle of the triplet $(\tilde
J_{a,b}(F),\pi_1^*\A,\pi_2^*\B)$.
\end{enumerate}
\end{lemma}

\begin{proof} 
Part (i) comes immediately from Lemma \ref{propjumping}(ii), while
part (ii) comes from the definition of
$\Sigma_{a,b}(F)$ and $J_{a,b}(F)$. 

For part (iii), observe that there is a commutative diagram
$$\begin{matrix}
S\otimes H^0(\O_{\P^n}(1))^*&\to&{T'_0}^*\cr
\downarrow&&\downarrow\cr
H^0(\tilde J_{a,b}(F),\pi_1^*\A)\otimes
H^0(\tilde J_{a,b}(F),\pi_2^*\B)&\to&
H^0(\tilde J_{a,b}(F),\pi_1^*\A\otimes\pi_2^*\B)
\end{matrix}$$
in which:

--The top map is the dual of the inclusion ${T'_0}^*\to
S^*\otimes H^0(\O_{\P^n}(1))$, which is naturally identified with the
map 
$$H^0(G(a,S^*),\A)\otimes
H^0(G(b,H^0(\O_{\P^n}(1)),\B)\to H^0(G(ab,{T'_0}^*),\Q)$$
consisting of the restriction from $G(ab,S^*\otimes
H^0(\O_{\P^n}(1)))$ to $G(ab,{T'_0}^*)$ of the sections of the
universal quotient bundle of rank $ab$.

--The vertical maps are, with the above identifications,
$\alpha\otimes\beta$ and $\gamma$, so that they are isomorphisms by
hypothesis.  

--The bottom map is the multiplication map whose dual, by Example
\ref{mainconstruction}, defines (in the sense of Lemma \ref{data})
the Schwarzenberger bundle of the triplet
$(\tilde J_{a,b}(F),\pi_1^*\A,\pi_2^*\B)$.

Since the dual of the top map is the one defining (in the sense of
Lemma \ref{data}), the bundle $F_0$, part (iii) follows from the
vertical isomorphisms.
\end{proof}

\begin{example} \label{ejemploslineales} {\rm
We illustrate the above situation in the case $a=b=1$, the one on
which we will concentrate in this paper. In this case, $\tilde
J(F)$ is the intersection of the Segre variety 
$\P(S)\times{\P^n}^*$ with the projective space $\P(T_0)$. The
conditions of Lemma \ref{jumpingSegre}(iii) are the linear
normality and nondegeneracy, respectively, of $\tilde J(F)$ in
$\P(T_0)$, of $\Sigma(F)$ in $\P(S)$, and of $J(F)$ in
${\P^n}^*$. Using the standard properties of the classical Segre
embedding, we will have the following properties that we will use
frequently:
\begin{enumerate}
\item[(i)] The set $\tilde J(F)$ is cut out by quadrics.
\item[(ii)] The fibers of $\pi_1,\pi_2$ are linear subspaces
of $\P(T_0)$.
\item[(iii)] Any linear subspace of $\tilde J(F)$ is
contained in a fiber of $\pi_1$ or $\pi_2$.
\end{enumerate}

Depending on the context, we
will regard $\tilde J(F)$ as a subvariety of the projective
space $\P(T_0)$ or as a subvariety of the product
$\P(S)\times{\P^n}^*$. It will be useful to observe that the
relation among these two points of view is that the hyperplane
section of $\tilde J(F)$ as a subvariety of
$\P(T_0)$ is $\pi_1^*\O_{\P(S)}(1)\otimes\pi_2^*\O_{{\P^n}^*}(1)$,
where $\pi_1,\pi_2$ are the projections to $\P(S)$ and ${\P^n}^*$.
}\end{example}

\begin{remark} \label{pocosjumping} {\rm 
Observe that, in general, one should not expect the hypothesis of
Lemma \ref{jumpingSegre}(iii) to hold. This is because the condition
(ii) in Lemma \ref{data} is open in the set of linear maps
$\varphi:T^*\to S^*\otimes H^0(\O_{\P^n}(1))$. Hence a general
$\varphi$ will produce a Steiner bundle, which will also be
reduced. Since $G(a,S^*)\times G(b,H^0(\O_{\P^n}(1)))$ tends to
have a big codimension in
$G(ab,S^*\otimes H^0(\O_{\P^n}(1)))$, one should expect its
intersection with a general $G(ab,T^*)$ to be very small, and in
general empty. Therefore, for arbitrary big values of $s,t,a,b$,
the set $\tilde J_{a,b}(F)$ is expected to be empty, i.e. a
general Steiner bundle will not have jumping
$(a,b)$-subspaces. 

For example, if $s=3,t=n+4$, a general $(3,n+4)$-Steiner bundle
on $\P^n$ does not have jumping hyperplanes when $n\ge4$, since
the Segre variety $\P^2\times\P^n$ has codimension $2n$ in
$\P^{3n+2}$, so its intersection with a general linear space of
dimension $n+3$ is empty. This also shows that, for
$n=2$, the set of jumping pairs of a general $F$ is a curve
in $\P^2$, so that $F$ cannot be the Schwarzenberger bundle of the
triplet $(\P^2,\O_{\P^2}(1),\O_{\P^2}(1))$ (see Example
\ref{Veronese}). However, we will see in Theorem \ref{teoremon}(iv) that,
when $s=2$, the expected dimension of the set of jumping
pairs is ``the right one''.  
}\end{remark}

Our goal now is to see that the hypothesis of Lemma
\ref{jumpingSegre}(iii) holds if $F$ has ``many'' jumping pairs. The first
thing we will need to do is to understand how big the dimension of
$\tilde J(F)$ can be. By Example \ref{ejemploslineales}, we need  to study
how the Segre variety can intersect linear subspaces of given
dimension. To do so, we need a technical result of linear algebra (in
which it is crucial that the ground field is algebraically closed),
which we state as a separate lemma. Even if we are going to use it
only for $a=b=1$, we include the general statement, since the general
proof does not add any difficulty and since we hope that it could be
useful in a future work.

\begin{lemma} \label{tecnico} 
Let $U,V$ be two vector spaces of respective dimensions $r,s$ over
the algebraically closed field $k$. Fix nonzero subspaces
$B\subset U$ of codimension $b<r$ and $A\subset V$ of dimension
$a<s$. Let $W$ be a $t$-dimensional linear space of $\Hom(U,V)$
such that for any $u\in U$ and any $v\in V$ there exists $f\in W$
such that $f(u)=v$.  Then $$\dim\{f\in W\ |\ f(B)\subset A\}\le
t-r-s+a+b+1.$$ 
\end{lemma}

\begin{proof} 
We take any basis $v_1,\dots,v_s$ of $V$ such that $v_1,\dots,
v_a\in A$ and pick also any nonzero vector $u_1\in B$. By
assumption, there exist linear maps
$g_{a+1},\dots,g_s$ in $W$ such that $g_i(u_1)=v_i$ for
$i=a+1,\dots,s$. 

Let us construct next, for $i=2,\dots, r-b$, vectors
$u_1,\dots,u_i\in B$ and maps $h_2,\dots,h_i\in W$ such
that 
$$h_i(u_i)\not\in<v_1,\dots,v_a,
g_{a+1}(u_i),\dots,g_s(u_i),h_2(u_i),\dots,h_{i-1}(u_i)>\hbox{ for
}i=2,\dots, r-b.$$

We do it by iteration, so we can assume that we
have already constructed
$u_1,\dots,u_{i-1}$ and $h_2,\dots,h_{i-1}$. Take any
$u'_i\in B\setminus<u_1,\dots,u_{i-1}>$ (we can do so because
$i-1\le r-b-1<\dim B$). For any
$\lambda_1,\dots,\lambda_i$, consider the vectors
$$v_1,\dots,v_a,
g_{a+1}(\lambda_1u_1+\dots+\lambda_{i-1}u_{i-1}+\lambda_iu'_i),\dots,
g_s(\lambda_1u_1+\dots+\lambda_{i-1}u_{i-1}+\lambda_iu'_i),$$
$$h_2(\lambda_1u_1+\dots+\lambda_{i-1}u_{i-1}+\lambda_iu'_i),\dots,
h_{i-1}(\lambda_1u_1+\dots+\lambda_{i-1}u_{i-1}+\lambda_iu'_i)$$ 
and the $(s+i-2)\times s$ matrix given by their coordinates
with respect to $v_1\dots,v_s$. This matrix will have no maximal
rank if and only if the $(s-a+i-2)\times(s-a)$ submatrix obtained
by removing the first $a$ rows and columns has no maximal rank.
The assumption $s>a$ implies that this submatrix is not vacuous,
and since its entries are linear forms in
$\lambda_1,\dots,\lambda_i$ and  the ground field is
algebraically closed, there exists some nonzero solution
$\lambda_1,\dots,\lambda_i$ for which the submatrix has not
maximal rank. We take
$u_i=\lambda_1u_1+\dots+\lambda_{i-1}u_{i-1}+\lambda_iu'_i$ for
some nonzero solution as above. Hence there exists $v\in
V\setminus<v_1,\dots,v_a,
g_{a+1}(u_i),\dots,g_s(u_i),h_2(u_i),\dots,h_{i-1}(u_i)>$. We
thus take $h_i\in W$ such that $h_i(u_i)=v$, which completes the
iteration process.

Assume that we know that $g_{a+1},\dots,g_s,h_2,\dots,h_{r-b}\in W$
are linearly independent modulo $\{f\in W\ |\ f(B)\subset
A\}$. This would imply that, inside the vector space $W$, the
subspace $\{f\in W\ |\ f(B)\subset A\}$ has zero
intersection with the $(r+s-a-b-1)$-dimensional subspace generated
by $g_{a+1},\dots,g_s,h_2,\dots,h_{r-b}$. We would get then the
wanted inequality.

We are thus left to prove that $g_{a+1},\dots,g_s,h_2,\dots,h_{r-b}\in W$
are linearly independent modulo $\{f\in W\ |\ f(B)\subset
A\}$. Assume that we have some linear combination
$$f:=\mu_{a+1}g_{a+1}+\dots+\mu_sg_s
+\nu_2h_2+\dots+\nu_{r-b}h_{r-b}$$
such that $f(B)\subset A=<v_1,\dots,v_a>$.
Applying both terms to $u_{r-b}$, we get
$$\nu_{r-b}h_{r-b}(u_{r-b})\in
<v_1,\dots,v_a,g_{a+1}(u_{r-b}),\dots,g_s(u_{r-b}),
h_2(u_{r-b}),\dots,h_{r-b-1}(u_{r-b})>,$$
which implies $\nu_{r-b}=0$, by our choice of $u_{r-b}$.
Knowing this vanishing, we consider now
$f(u_{r-b-1})$ and get 
$\nu_{r-b-1}=0$ in the same way, and iterating we get
$\nu_2=\dots=\nu_{r-b}=0$. We thus
have
$f(u_1)=\mu_{a+1}v_{a+1}+\dots+\mu_sv_s$, which implies now
$\mu_{a+1},\dots,\mu_s=0$ since $f(u_1)\in<v_1,\dots,v_a>$.
\end{proof}

We can now give, for $a=b=1$, an upper bound for the dimension
of the set of jumping pairs. Since Lemma \ref{jumpingSegre} gives
$J(F)=\pi_2(\tilde J(F))$, the same bound will hold for the
dimension of the set of jumping hyperplanes. Observe that our
bound is sharp, because it is achieved in the cases of Examples
\ref{schwarzenberger}, \ref{superficieminimal},
\ref{rationalnormalscroll} and
\ref{Veronese}, (since at least the points of $X$ provide jumping
pairs).

\begin{theorem} \label{teoremon}
Let $F$ be an $(s,t)$-Steiner bundle on $\P^n$ with $s\ge2$. Then:
\begin{enumerate}
\item[(i)] The embedded Zariski tangent space at any point of
$\tilde J(F)$ has dimension at most $t-n-s+1$; in particular, 
$\dim\tilde J(F)\le t-n-s+1$.
\item[(ii)] If $\tilde J\subset\P(S)\times{\P^n}^*$ is a
component of $\tilde J(F)$ such that its projection to $\P(S)$
or ${\P^n}^*$ is constant, then $\dim\tilde J<t-n-s+1$.
\item[(iii)] If $\tilde J(F)$ has dimension $t-n-s+1$, then 
$F$ is reduced and $\tilde J(F)$ is smooth at the points of any
of its irreducible components of maximal dimension.
\item[(iv)] If $s=2$ and $F$ is reduced, then $\tilde J(F)$
is a rational normal scroll of dimension $t-n-1$ (and degree
$n+1$) and $F$ is the Schwarzenberger bundle of Example
\ref{rationalnormalscroll}.
\item[(v)] If $n=1$ and $F$ is reduced, then $\tilde J(F)$
is a rational normal scroll of dimension $t-s$ (and degree
$s$) and $F$ is the Schwarzenberger bundle of Example
\ref{superficieminimal}.
\end{enumerate}
\end{theorem}

\begin{proof} 
To prove (i), we identify $\P(S\otimes H^0(\O_{\P^n}(1))^*)$ with
the set of nonzero linear maps (up to multiplication by a constant)
$H^0(\O_{\P^n}(1))^*\to S^*$. Then the Segre variety
corresponds to maps of rank one, while
$\P(T_0)$ corresponds to the subspace
$T_0^*\subset\Hom(H^0(\O_{\P^n}(1))^*,S^*)$ of Lemma \ref{reduced}. 
Fix any point $(\alpha,H)\in\tilde
J(F)\subset\P(S)\times{\P^n}^*$. As a point in $\P(S\otimes
H^0(\O_{\P^n}(1))^*)$, it is represented by a linear map
$H^0(\O_{\P^n}(1))^*\to S^*$ whose kernel is a hyperplane $\vec
H\subset H^0(\O_{\P^n}(1))^*$ defining $H$ and whose image is a
line $A\subset S^*$ representing $\alpha$. The embedded
tangent space to the Segre variety at $(\alpha,H)$ corresponds
then to the subspace of linear maps $f:H^0(\O_{\P^n}(1))^*\to S^*$
such that $f(\vec H)\subset A$ (see for instance \cite{H} Example
14.16). Since
$\tilde J(F)$ is the intersection of the Segre variety with
$\P(T_0)$, it follows that its embedded tangent space at
$(\alpha,H)$ corresponds to the subspace of linear maps $f\in
T_0^*$ such that $f(\vec H)\subset A$. By Lemma \ref{tecnico}
(whose hypotheses are satisfied by Lemma
\ref{data}), this subspace has dimension at most
$t_0-(n+1)-s+3$, where $t_0=\dim T_0$. Since $t_0\le t$, it
follows that the dimension of the embedded Zariski tangent space
of $\tilde J(F)$ at $(\alpha,H)$ is at most $t-n-s+1$, which
completes the proof of (i).

In order to prove (ii), assume first that the image of
$\tilde J$ in $\P(S)$ is a point corresponding to a line $A\subset
S^*$. Then the embedded tangent space at any point of $\tilde J$ is
contained in the subspace corresponding to the linear maps
$f\in T_0^*$ such that $f(H^0(\O_{\P^n}(1))^*)\subset A$. By Lemma
\ref{tecnico} (taking $B=H^0(\O_{\P^n}(1))^*$), we get, arguing as
in (i), that the embedded tangent space would have dimension at
most $t-n-s$, as wanted. If instead the image of $\tilde J$ in
${\P^n}^*$ is an element corresponding to a hyperplane $B\subset
H^0(\O_{\P^n}(1))^*$, we proceed in the same way: now the embedded
tangent space of $\tilde J$ is contained in the subspace
corresponding to the linear maps $f\in T_0^*$ such that
$f(B)=0$, and we use Lemma \ref{tecnico} taking $A=0$.

To prove (iii), assume that we have $\dim\tilde J(F)=t-n-s+1$.
Hence in the  proof of (i) all inequalities are equalities. In
particular $t_0=t$, so that $F$ is reduced. On the other
hand, for any component of $\tilde J(F)$ of dimension
$t-n-s+1$, the dimension of its embedded tangent space at any
point cannot exceed $t-n-s+1$, by (i), so that all
the points of that component are smooth.

Assume now $s=2$ in order to prove (iv). Then
$\P(S)\times{\P^n}^*$ has codimension $n$ in $\P(S\otimes
H^0(\O_{\P^n}(1))^*)$, so that its intersection with $\P(T)$ has
dimension at least $t-1-n$. By (iii), it follows that $\tilde J(F)$
is a smooth complete intersection of $\P(S)\times{\P^n}^*$ and
$\P(T)$, i.e. a smooth rational normal scroll $\tilde
J(F)\subset\P(T)$ of dimension $t-n-1$, so that we can make the
identification $T=H^0(\O_{\tilde J(F)}(h))$, where $h$ is the
hyperplane section class of the scroll. It also follows from (ii)
that the projection $\pi_1:\tilde J(F)\to\P(S)=\P^1$ is not constant,
hence it is surjective. Therefore all the fibers of
$\pi_1$ (which are linear spaces, by Example \ref{ejemploslineales}(ii))
have dimension $t-n-2$, so that $\pi_1$ gives the scroll
structure on $\tilde J(F)$. We can thus identify
$S=H^0(\O_{\tilde J(F)}(f))$, where $f$ is the class of a fiber
of the scroll and, as pointed out in Example \ref{ejemploslineales},
the map from $\tilde J(F)$ to ${\P^n}^*$ is given by $\O_{\tilde
J(F)}(h-f)$. In order to complete the proof of (iv) we need to show,
by Lemma \ref{jumpingSegre}(iii), that we can identify
$H^0(\O_{\P^n}(1))^*=H^0(\O_{\tilde J(F)}(h-f))$. This identification
comes from the fact that the restriction map
$H^0(\O_{\P(S)\times{\P^n}^*}(0,1))\to H^0(\O_{\tilde J(F)}(h-f))$ is
an isomorphism because $\tilde J(F)$ is the complete
intersection of $\P(S)\times{\P^n}^*$ and a linear space.

Finally, (v) was proved in Example \ref{superficieminimal}
(observe that a Steiner bundle on $\P^1$ is reduced if and only if
it is ample), although the same proof as in (iv) holds.
\end{proof}

\begin{remark} \label{solojumping} {\rm
Observe that part (iv) of Theorem \ref{teoremon} is giving more
information about Example \ref{rationalnormalscroll}. Indeed our
proof shows that we have $X=\tilde J(F)$, even with the scheme
structure of $\tilde J(F)$ as intersection of the Segre variety
and a linear space, and shows in particular that any jumping
hyperplane of
$F$ is coming from a point of $X$. Hence, for the
Schwarzenberger bundles of Example \ref{rationalnormalscroll}, we get
a positive answer to Question \ref{otrapregunta} (the same holds
for Example \ref{superficieminimal}). Incidentally, observe that, in
this example, the set of jumping hyperplanes has not always maximal
dimension $t-n-1$. This is because
$J(F)$ is the image of the rational normal scroll $X$ via
$\O_X(h-f)$, which drops dimension if (and only if) $X$ is
the Segre variety $\P^1\times \P^n$ (which is equivalent to say
$t=2n+2$), in which $\O_X(h-f)$ induces the projection onto
$\P^n$. In particular, in this last case, all the hyperplanes are
jumping hyperplanes.

Observe also that, in general, the answer to Question
\ref{otrapregunta} can be negative. For example, if $X$
is an elliptic curve and $L,M$ are line bundles on $X$ of
respective degrees $2$ and $n+1$, the Schwarzenberger bundle of
the triplet $(X,L,M)$ is a $(2,n+3)$-Steiner bundle $F$.
However, Theorem \ref{teoremon}(iv) implies that $\tilde J(F)$ and
$J(F)$ are rational normal scrolls of dimension two instead of
just the original elliptic curve $X$ (it can be seen that these
scrolls consist of the union of the lines spanned by the pairs of
points of $X$ given by the divisors in the linear system defined by
$L$). 
}\end{remark}

\section{Steiner bundles with jumping locus of
maximal dimension}

In this section we will characterize $(s,t)$-Steiner
bundles for which $\tilde J(F)$ has the maximal dimension
$t-n-s+1$, showing that they are exactly Examples
\ref{schwarzenberger}, \ref{superficieminimal},
\ref{rationalnormalscroll} and
\ref{Veronese} (hence we give a positive answer to Question
\ref{question} in this case). When the maximal dimension is one (i.e.
when $t=n+s$ or, equivalently, $F$ has rank $n$), we recover the
known result that Steiner bundles of rank $n$ with a curve of
jumping hyperplanes are precisely the classical Schwarzenberger
bundles (see \cite{V1}). 

The main idea, borrowed from the case of rank $n$, will be to
produce, from a given $(s,t)$-Steiner bundle, an
$(s-1,t-1)$-Steiner bundle (thus with the same rank of $F$) with
essentially the same jumping hyperplanes. Then, after an iteration,
we will eventually we arrive to a Steiner bundle with $s=2$ to which
we can apply Theorem \ref{teoremon}(iv). Analogously, we will produce an
$(s,t-1)$ Steiner bundle on a (jumping) hyperplane, and eventually
arrive to a Steiner bundle on $\P^1$ to which we can apply Theorem
\ref{teoremon}(v) (we will omit the details of this second iteration,
stating the results we will need in Remark
\ref{simetria}).

The starting point is the following (see \cite{V1} Proposition 2.1 for the
case of rank $n$):

\begin{proposition} \label{transform} 
Let $F$ be a reduced $(s,t)$-Steiner bundle on $\P^n$, and let
$\pi_1,\pi_2$ denote the two projections from $\tilde
J(F)\subset\P(S)\times{\P^n}^*$. Let $(\alpha,H)$ be a jumping
pair of $F$, let $i:S'\subset S$ and $j:T'\subset T$ be the hyperplane
inclusions corresponding respectively to $\alpha\in\P(S)$ and 
$(\alpha,H)\in\P(T)$. If $F'$ is the kernel of the natural composition
$F\to F_{|H}\to\O_H$ defined by $(\alpha,H)$ (see Lemma
\ref{propjumping}(iii)) then:
\begin{enumerate}
\item[(i)] $F'$ is an $(s-1,t-1)$-Steiner bundle $F'$ having a
resolution
$$0\to S'\otimes\O_{\P^n}(-1)\to T'\otimes\O_{\P^n}\to
F'\to0.$$
\item[(ii)] The linear map $\varphi'$ defining $F'$ (see Lemma
\ref{data}) fits in a commutative diagram
$$\begin{matrix}
T^*&\mapright{\varphi}&S^*\otimes H^0(\O_{\P^n}(1))\cr
\ \ \ \downarrow{j^*}&&\downarrow{i^*\otimes id}\cr
{T'}^*&\mapright{\varphi'}&{S'}^*\otimes H^0(\O_{\P^n}(1))
\end{matrix}$$
\item[(iii)] $J(F)\subset J(F')\cup\pi_2\pi_1^{-1}(\alpha)$.
\end{enumerate}
\end{proposition}

\begin{proof} 
We have the following commutative diagram
$$\begin{matrix}
&&0&&0&&0&\cr
&&\downarrow&&\downarrow&&\downarrow&\cr
0&\to&S'\otimes\O_{\P^n}(-1)&\to&T'\otimes\O_{\P^n}&\to&F'&\to&0\cr
&&\downarrow&&\downarrow&&\downarrow&\cr
0&\to&S\otimes\O_{\P^n}(-1)&\to&T\otimes\O_{\P^n}&\to&F&\to&0\cr
&&\downarrow&&\downarrow&&\downarrow&\cr
0&\to&\O_{\P^n}(-1)&\to&\O_{\P^n}&\to&\O_H&\to&0\cr
&&\downarrow&&\downarrow&&\downarrow&\cr
&&0&&0&&0&
\end{matrix}$$
where the first column is defined by the quotient of $S$
corresponding to $\alpha$, the second column is defined by the
quotient of $T$ corresponding to $(\alpha,H)$, and the first row
is defined as a kernel. This proves (i).

Taking duals, we get another commutative diagram
$$\begin{matrix}
0&\to&F^*&\to&T^*\otimes\O_{\P^n}&\to&S^*\otimes\O_{\P^n}(1)&\to&0\cr
&&\downarrow&&\downarrow&&\downarrow&\cr
0&\to&{F'}^*&\to&{T'}^*\otimes\O_{\P^n}&\to&{S'}^*\otimes\O_{\P^n}(1)&\to&0\cr
\end{matrix}$$
which, taking cohomology, produces (ii). 

To prove (iii), consider any jumping hyperplane $H_1$ of $F$  and
assume it is not in $\pi_2\pi_1^{-1}(\alpha)$, so that it comes from a
jumping pair $(\alpha_1,H_1)$ with $\alpha_1\ne\alpha$. This jumping
pair is represented by a nonzero tensor $v_1\otimes h_1\in S^*\otimes
H^0(\O_{\P^n}(1))$ in the image of $\varphi$ (where $h_1$ is an
equation of $H_1$). Since $\alpha_1\ne\alpha$, it follows that
$i^*(v_1)\otimes h_1$ is nonzero, and it is also in the image of
$\varphi'$, by (ii). This implies that $([i^*(v_1)],H_1)$ is a
jumping pair of $F'$, so that $H_1$ is a jumping hyperplane of $F'$,
as wanted.
\end{proof}

\begin{remark} \label{dificultaditeracion} {\rm
The idea now is that, when performing the iteration process, part
(iii) of Proposition \ref{transform} should provide enough
information to keep track the set of jumping pairs until we arrive
to a Steiner bundle with $s=2$. There are two difficulties to do
so. First of all, some bundle in the iteration process could be
non reduced, although we could deal with this taking its reduced
summand and using Lemma \ref{propjumping}(i). The main difficulty
is however that Proposition \ref{transform}(iii) does not relate
$J(F)$ and $J(F')$ if $J(F)$ is contained in some
$\pi_2\pi_1^{-1}(\alpha)$. Of course this behavior seems very
unlikely (for instance, it does not hold if
$\dim\tilde J(F)=t-n-s+1$, as Theorem \ref{teoremon}(ii) guarantees), and
we could impose that it does not hold for our original $F$, but still
it could hold for some other Steiner bundle in the iteration process. 

In the case of Steiner bundles of rank $n$ (the one studied in
\cite{V1}), which are always reduced, this last difficulty can be avoided as
follows. Any Steiner bundle $F'$ in the process has rank $n$, so that
from Theorem \ref{teoremon}(i) its set of jumping hyperplanes has dimension
at most one. Therefore, if the projection $\pi'_1:\tilde J(F')\to
\P(S')$ were constant, its fiber (which is a linear space, by
Example \ref{ejemploslineales}(ii)) would be either a point or a line. It
cannot be a line by Theorem \ref{teoremon}(ii), so that necessarily
$F'$ would have only one jumping hyperplane. This is the
key underlying idea in \cite{V1} that allows even to limit the number
of jumping hyperplanes when there are finitely many.
}\end{remark}

The key to deal with the first difficulty of Remark
\ref{dificultaditeracion} is the following (in which we also pay
attention to jumping pairs instead of just jumping hyperplanes):

\begin{proposition} \label{proptransform}
In the situation of Proposition \ref{transform}, set
${T'_0}^*:={\rm Im}\varphi'$ and let
$F'=F'_0\oplus(T'/T'_0)\otimes\O_{\P^n}$ be the decomposition of
Lemma \ref{reduced}. Then:
\begin{enumerate}
\item[(i)] The projection from the linear subspace
$\pi_1^{-1}(\alpha)\subset\P(T)$ is the map 
$pr_{(\alpha,H)}:\P(T)\to\P(T'_0)$ induced by the composition
$T\mapright{j^*}T'\mapright{\varphi'}T'_0$. In particular, $\dim
T'_0=t-1-\dim\pi_1^{-1}(\alpha)$.
\item[(ii)] If $pr_\alpha:\P(S)\to\P(S')$ denotes the projection
from $\alpha$, for any $(\alpha_1,H_1))\in\tilde J(F)$ with
$\alpha_1\ne\alpha$, we have the equality
$$pr_{(\alpha,H)}(\alpha_1,H_1)=(pr_\alpha(\alpha_1),H_1)$$ 
and this is a jumping pair of $F'$ and $F'_0$.
\item[(iii)] $\tilde J(F'_0)$ contains the image under
$pr_{(\alpha,H)}$ of any
component of $\tilde J(F)\subset\P(T)$ not contained in 
$\pi_1^{-1}(\alpha)$.
\item[(iv)] $\Sigma(F'_0)$ contains the image under
$pr_\alpha$ of any
component of $\Sigma(F)\subset\P(S)$ different from
$\{\alpha\}$.
\end{enumerate}
\end{proposition}

\begin{proof} 
It follows readily from the commutative diagram of Proposition
\ref{transform}(ii). For example, part (i) comes from the fact that the
subspace of $T^*$ corresponding to $\pi_1^{-1}(\alpha)$ is the
kernel of $(i^*\otimes id)\circ\varphi=\varphi'\circ j^*$. Part
(ii) is now the interpretation of the diagram of Proposition
\ref{transform}(ii) (recall that $F'$ and $F'_0$ has the same jumping
pairs, by Lemma \ref{propjumping}(i)). Finally, parts (iii) and (iv) are
proved from (ii) (in fact, it is the same proof as the one of of
Proposition \ref{transform}(iii)).
\end{proof}

The next proposition shows that, for Steiner bundles of arbitrary
rank, the second difficulty of Remark \ref{dificultaditeracion} can
be overcome with the same ideas as in the case of rank $n$ if we
assume that the set of jumping pairs has the maximal dimension
allowed by Theorem \ref{teoremon}(i) (observe that, in this case,
the bundle is necessarily reduced, by Theorem
\ref{teoremon}(iii)).

\begin{proposition} \label{maximaltransform}
Let $F$ be an $(s,t)$-Steiner bundle on $\P^n$ with $s\ge2$ and
such that
$\tilde J(F)$ has  dimension $t-n-s+1$. Let $\tilde J_0$ be a
component of $\tilde J(F)$ of maximal dimension and fix
$(\alpha,H)\in\tilde J_0$. Then, if $F'$ is the Steiner bundle
constructed in Proposition \ref{transform} and $F'_0$ is its
reduced part, the following hold:
\begin{enumerate}
\item[(i)] The image of both $\tilde J_0$ and $\tilde J(F)$ under
the projection $pr_{(\alpha,H)}$ from $\pi_1^{-1}(\alpha)$ has
dimension  $t-n-s+1-\dim\pi_1^{-1}(\alpha)$.
\item[(ii)] $\tilde J(F'_0)$ has dimension
$t-n-s+1-\dim\pi_1^{-1}(\alpha)$.
\item[(iii)] If $\tilde J(F'_0)$ is irreducible, then:
\begin{enumerate} 
\item[a)] $\tilde J(F'_0)$ is the image of $\tilde J(F)$ under
the projection $pr_{(\alpha,H)}$ from $\pi_1^{-1}(\alpha)$.
\item[b)] $\tilde J(F)$ is irreducible.
\item[c)] $J(F)=J(F'_0)$.
\item[d)] $\Sigma(F'_0)$ is the image of $\Sigma(F)$ under the
inner projection $pr_\alpha$ from $\alpha$.
\end{enumerate}
\end{enumerate}
\end{proposition}

\begin{proof} 
Since, by Theorem \ref{teoremon}(i), $\tilde J(F'_0)$ has dimension at
most $\dim T'_0-n-(s-1)+1$ and, by Proposition
\ref{proptransform}(i), $\dim T'_0=t-1-\dim\pi_1^{-1}(\alpha)$,
part (i) will follow if we prove that the image of $\tilde J_0$
under $pr_{(\alpha,H)}$ has dimension at least
$t-n-s+1-\dim\pi_1^{-1}(\alpha)$. Assume by contradiction that
$\tilde J_0$ drops dimension by $\dim\pi_1^{-1}(\alpha)+1$ when
projecting from $\pi_1^{-1}(\alpha)$. This means that $\tilde J_0$ is
a cone with vertex $\pi_1^{-1}(\alpha)$. Since any line in the cone
is contained in a fiber of $\pi_1$ or $\pi_2$ (Example
\ref{ejemploslineales}(iii)), it follows that $\tilde J_0$ is contained in
$\pi_1^{-1}(\alpha)\cup\pi_2^{-1}(H)$. But
$\tilde J_0$ is irreducible, so that it contained in
$\pi_1^{-1}(\alpha)$ or $\pi_2^{-1}(H)$, which  contradicts
Theorem \ref{teoremon}(ii). 

To prove (ii), we have, on one hand, that Proposition
\ref{proptransform}(iii) implies that $\tilde J(F')$ contains the
image of $\tilde J_0$ under $pr_{(\alpha,H)}$, which has dimension
$t-n-s+1-\dim\pi_1^{-1}(\alpha)$, by (i). On the other hand, Theorem
\ref{teoremon}(i) implies $\dim J(F')\le
t-n-s+1-\dim\pi_1^{-1}(\alpha)$, so that (ii) follows. 

To prove (iii), observe first that $\tilde J(F)$ cannot have any
component contained in $\pi_1^{-1}(\alpha)$. Indeed
$\pi_1^{-1}(\alpha)$ is contained in $\tilde J_0$, since otherwise it
would be contained in another component of $\tilde J(F)$. But then
such a component would meet $\tilde J_0$ at least at the point
$(\alpha,H)$, implying that $(\alpha,H)$ is a singular point of
$\tilde J(F)$, contradicting Theorem \ref{teoremon}(iii).

I claim now that $\tilde J(F')$ coincides with the image  of both
$\tilde J_0$ and $\tilde J(F)$ under $pr_{(\alpha,H)}$. Indeed, both
images are contained in $\tilde J(F')$ by Proposition
\ref{proptransform}(iii) (and the above observation), and on the
other hand they have dimension $t-n-s+1-\dim\pi_1^{-1}(\alpha)$, by
(i). Since, by (ii), $\tilde J(F')$ has also dimension
$t-n-s+1-\dim\pi_1^{-1}(\alpha)$, its irreducibility proves the
claim, and also part a).

To prove part b), assume for contradiction that $\tilde J(F)$ has
another component $\tilde J_1$ different from $\tilde J_0$, and fix
any point $(\alpha_1,H_1)\in\tilde J_1\setminus\tilde J_0$. By our
previous claim, the image of $(\alpha_1,H_1)$ under $pr_{(\alpha,H)}$
is also in the image of $\tilde J_0$. In particular, there is a
line $\Delta$ trisecant to $\tilde J(F)$, passing through
$(\alpha_1,H_1)$ and meeting $\pi_1^{-1}(\alpha)$. Since $\tilde
J(F)$ is cut out by quadrics (Example \ref{ejemploslineales}(i)),
it follows that $\Delta$ is contained in $\tilde J(F)$. But
$\Delta\not\subset\tilde J_0$, so that there is another component
of $\tilde J(F)$ containing $\Delta$. Therefore
$\tilde J_0$ meets that component at the point $(\alpha,H)$, so that
$(\alpha,H)$ is a singular point of $\tilde J(F)$ that is in $\tilde
J_0$. This contradicts once more Theorem \ref{teoremon}(iii), hence b)
holds.

We prove part c) by showing the double inclusion.
Observe first that the irreducibility of $\tilde J(F)$ implies the
irreducibility of  $J(F)$. Thus, Proposition \ref{transform}(iii)
implies, together with Theorem \ref{teoremon}(ii), that $J(F)$ is
contained in $J(F')$, which is $J(F'_0)$ by Lemma
\ref{propjumping}(i), so that we are left to prove the other
inclusion. Since $pr_{(\alpha,H)}(\tilde
J(F)\setminus\pi_1^{-1}(\alpha))$ is dense in
$\tilde J(F'_0)$, also $\pi'_1(pr_{(\alpha,H)}(\tilde
J(F)\setminus\pi_1^{-1}(\alpha)))$ is dense in $J(F'_0)$, so it is
enough to prove that any element of it is also in $J(F)$. We thus
take $H'\in J(F'_0)$ for which there exists $\alpha'\in\P(S')$
such that $(\alpha',H')=pr_{(\alpha,H)}(\alpha_1,H_1)$ for some
$(\alpha_1,H_1)\in\tilde J(F)$ with $\alpha_1\ne\alpha$.
Since $pr_{(\alpha,H)}(\alpha_1,H_1)=(pr_\alpha(\alpha_1),H_1)$ by
Proposition \ref{proptransform}(ii), it follows $H'=H_1$, hence $H'\in
J(F)$, as wanted. 

Finally, part d) is proved also by double inclusion. First, observe
that $\Sigma(F)$ is irreducible by b), so that it cannot be just
$\{\alpha\}$ by Theorem \ref{teoremon}(ii). Therefore, Proposition
\ref{proptransform}(iv) implies that $\Sigma(F'_0)$ contains the
image of $\Sigma(F)$ under $pr_\alpha$. Reciprocally, take any
$\alpha'\in\Sigma(F'_0)$. As before, we can assume that there exists
$H'\in J(F'_0)$ such that $(\alpha',H')=pr_{(\alpha,H)}(\alpha_1,H_1)$
for some $(\alpha_1,H_1)\in\tilde J(F)$ with $\alpha_1\ne\alpha$.
Hence Proposition \ref{proptransform}(ii) implies
$\alpha'=pr_\alpha(\alpha_1)$. Since obviously
$\alpha_1\in\Sigma(F)$, the result follows.
\end{proof}

\begin{remark} \label{simetria} {\rm
Exactly in the same way as in Proposition \ref{transform}, one
could construct from $F$ and a jumping pair $(\alpha,H)$ the
Steiner bundle defined by ${T'}^*\to S^*\otimes H^0(\O_H(1))$. This
time we get an $(s,t-1)$-Steiner bundle $F'$ on
$H$ and the same results of this section hold by permuting the roles
of $J(F)$ and $\Sigma(F)$. In particular, if $\tilde J(F)$ has the
maximal dimension allowed by Theorem \ref{teoremon}(i), then also $\tilde
J(F')$ has the maximal dimension allowed by Theorem
\ref{teoremon}(i); and if $\tilde J(F')$ is irreducible, then
$\Sigma(F)=\Sigma(F')$. We will not prove it, since it is done
exactly in the same way.
}\end{remark}

Before stating and proving our main result, we include, for
the reader's convenience, the following easy lemma about varieties
of minimal degree that we will need. By variety of minimal
degree we mean a nondegenerate irreducible variety in a projective
space such that its degree minus its codimension is one. We recall
(see for example \cite{H} Theorem 19.9) that a smooth variety of minimal
degree is either a quadric, a rational normal scroll (this includes
the whole projective space and rational normal curves) or a Veronese
surface in $\P^5$. 

\begin{lemma} \label{minimaldegree} 
Let $X\subset\P^N$ be a proper smooth irreducible projective
variety that is cut out by quadrics. Assume that $X$ contains an
$r$-dimensional linear subspace $\Lambda$ such that the projection
of $X$ from $\Lambda$ is a subvariety $X'\subset\P^{N-r-1}$ of
minimal degree with $\dim X'=\dim X-r$. Then also $X$ is a variety
of minimal degree.
\end{lemma}

\begin{proof} 
The inequality $\dim X'>\dim X-r-1$ implies $X$ is not a cone with
vertex $\Lambda$, so that there exists a point $x\in\Lambda$ such
that the line spanned by $x$ and a general point of $X$ is not
contained in $X$. Since $X$ is cut out by quadrics, such a line
cannot be  trisecant to $X$, and hence the projection from $x$
sends $X$ birationally to some $X_1\subset\P^{N-1}$. Therefore
both the degree and codimension of $X_1$ drop by one with respect
to those of $X$ (recall that $x$ is, by hypothesis, a smooth point
of $X$), and thus $X$ is a variety of minimal degree if and only
if $X_1$ is.

On the other hand, if $\Lambda_1$ is the $(r-1)$-dimensional image of
$\Lambda$, then $X'$ is the image of $X_1$ under the linear
projection from $\Lambda_1$. Since $\dim X'=\dim X_1-r$, this means
that $X_1$ is a cone with vertex $\Lambda_1$. Hence now $X'$ has the
same degree and codimension as $X_1$, so that $X_1$ is a variety of
minimal degree because $X'$ is. As observed before, this completes the
proof.
\end{proof}

\begin{theorem} \label{megateoremon}
Let $F$ be an $(s,t)$-Steiner bundle on $\P^n$ with $s\ge2$ and
such that $\tilde J(F)$ has  dimension $t-n-s+1$. Then $F$ is one
of the Schwarzenberger bundles of Examples \ref{schwarzenberger},
\ref{superficieminimal}, \ref{rationalnormalscroll} or
\ref{Veronese}.
\end{theorem}

\begin{proof} 
By Proposition \ref{maximaltransform}, we can construct an
$(s-1,t-1-\epsilon)$-Steiner bundle $F'_0$ such that $\tilde
J(F'_0)$ has dimension $t-n-s+1-\epsilon$. In particular, $F'_0$
has the maximal dimension allowed by Theorem
\ref{teoremon}(i). Iterating this process $s-2$ times, we arrive to a
reduced $(2,t'')$-Steiner bundle $F''$ such that $\tilde J(F'')$
has dimension $t''-n-s+1$. Thus Theorem \ref{teoremon}(iv)\ implies that
$\tilde J(F'')$ is a smooth rational normal scroll in $\P^{t''-1}$.
Since $\tilde J(F'')$ is irreducible, it follows from Proposition
\ref{maximaltransform}(iii) that also $\tilde J(F)$ is irreducible, that
$\tilde J(F'')$ is the image of $\tilde J(F)$ under a series of
$s-2$ inner projections from different linear subspaces, and that
$J(F)=J(F'')$. Since we know that $J(F'')$ is a rational normal
scroll, also $J(F)$ is. Similarly (see Remark \ref{simetria}), we can
produce from $F$ a reduced Steiner bundle $F'''$ on $\P^1$, so that it
follows from Theorem \ref{teoremon}(v) that $\Sigma(F)=\Sigma(F''')$ is a
rational normal scroll. On the other hand, Lemma
\ref{minimaldegree} implies that $\tilde J(F)$ is a variety of
minimal degree. Using the classification of smooth varieties of
minimal degree, we study separately each of the three
possibilities for $\tilde J(F)$ (we do not consider the
possibility of a quadric, since $\tilde J(F)$ has codimension
$n+s-2$, and this is one only in the case $n=1,s=2$, which is
trivial by Theorem \ref{teoremon}):

--If $\tilde J(F)$ is a rational normal curve (hence $t=n+s$) of
degree $t-1$, then necessarily $\tilde J(F'')$ is also a rational
normal curve obtained from $\tilde J(F)$ by projecting from $s-2$
points on it, so that $t''=t-s+2=n+2$. Therefore, Theorem
\ref{teoremon}(iv) says that $F''$ is the Schwarzenberger bundle of the
triplet $(\P^1,\O_{\P^1}(1),\O_{\P^1}(n))$, and in particular
$J(F'')$ is a rational normal curve of degree $n$. Since
$J(F)=J(F'')$, it follows that
$\pi_2^*\O_{{\P^n}^*}(1)=\O_{\P^1}(n)$. On the other hand, the
equality $\O_{\tilde J(F)}(1)=\O_{\P^1}(n+s-1)$ implies
$\pi_1^*\O_{\P(S)}(1)=\O_{\P^1}(s-1)$. The fact that $\tilde J(F)$,
$\Sigma(F)$ and $J(F)$ are rational normal curves implies that the
hypotheses of Lemma \ref{jumpingSegre}(iii) are satisfied, so that we are
in the case of Example \ref{schwarzenberger} (of course, this is
the case obtained in \cite{AO} and \cite{V1}, because we are dealing with
Steiner bundles of rank $n$).

--If $\tilde J(F)$ is a Veronese surface, then $t-n-s+1=2$ and $t=6$. 
An inner projection produces a rational normal scroll only when
projecting from one or two points, so that $s=3,4$. If $s=4$, then
$\tilde J(F'')$ is a smooth quadric in $\P^3$, so that $J(F'')$ is a
line. Since $J(F'')=J(F)$ and there are no regular maps from the
Veronese surface to $\P^1$, this case is not possible. Therefore
$s=3$ (hence $n=2$) and $\tilde J(F'')$ is a cubic surface
scroll in $\P^4$, so that $J(F'')$ is isomorphic to $\P^2$.
Since the map $\pi_2:\tilde J(F)\to J(F)$ has linear fibers, it
follows that it is an isomorphism and
$\pi_2^*\O_{{\P^n}^*}(1)\cong\O_{\P^2}(1)$. And since the hyperplane
class of $\tilde J(F)$ is $\O_{\P^2}(2)$, it also follows
$\pi_1^*\O_{{\P^n}^*}(1)\cong\O_{\P^2}(1)$ and $\pi_1$ is also
necessarily an isomorphism. By Lemma \ref{jumpingSegre}(iii), we
are in the case of Example \ref{Veronese}.

--Finally, assume that $\tilde J(F)\subset\P(T)$ is a rational normal
scroll of dimension $t-n-s+1>1$ (and degree $n+s-1$). 
Since the only non trivial splitting of the hyperplane section $h$ of
$\tilde J(F)$ into two globally generated line bundles is as
$$\O_{\tilde J(F)}(h)=\O_{\tilde J(F)}(rf)\otimes\O_{\tilde
J(F)}(h-rf)$$ 
for some integer $r>0$ (as usual, $f$ represents the fiber
of the scroll), one of the factors must be $\pi_1^*\O_{\P(S)}(1)$ and
the other one must be $\pi_2^*\O_{{\P^n}^*}(1)$. 

Assume for example $\pi_1^*\O_{\P(S)}(1)=\O_{\tilde J(F)}(rf)$ 
and $\pi_2^*\O_{{\P^n}^*}(1)=\O_{\tilde J(F)}(h-rf)$. In this case,
since $\tilde J(F)$, $\Sigma(F)$ and $J(F)$ are varieties of minimal
degree, Lemma \ref{jumpingSegre}(iii) implies that
$F$ is the Schwarzenberger bundle of the triplet $(\tilde
J(F),\O_{\tilde J(F)}(rf),\O_{\tilde J(F)}(h-rf))$. Hence 
$$s=h^0(\O_{\tilde J(F)}(rf))=r+1$$
$$n+1=h^0(\O_{\tilde J(F)}(h-rf))=t-r(t-n-s+1)$$
so that $t-n-s+1=t-(t-r(t-n-s+1)-1)-(r+1)+1$ and thus
$(r-1)(t-n-s)=0$, which implies $r=1$, so that we are in the case of
Example \ref{rationalnormalscroll}. 

The case $\pi_1^*\O_{\P(S)}(1)=\O_{\tilde J(F)}(h-rf)$ 
and $\pi_2^*\O_{{\P^n}^*}(1)=\O_{\tilde J(F)}(rf)$ is analogous, and
we would obtain here Example \ref{superficieminimal}. 
\end{proof}

If we just want to study the dimension of the set of jumping
hyperplanes, we have the following:

\begin{corollary} \label{corhiperplanos} 
Let $F$ be an $(s,t)$-Steiner bundle with $s\ge2$. Then $J(F)$ has
dimension at most $t-n-s+1$, with equality if and only if $F$ is
the Schwarzenberger bundle of one of the following triplets
$(X,L,M)$:
\begin{enumerate}
\item[(i)] $X=\P^1$, $L=\O_{\P^1}(s-1)$, $M=\O_{\P^1}(n)$.
\item[(ii)] $X\subset\P^{t-1}$ a smooth rational normal scroll of
dimension $t-n-1$ and degree $n+1$ different from $\P^1\times\P^n$
(i.e. $t\ne 2n+1$) and $L=\O_X(f)$, $M=\O_X(h-f)$ (see Example
\ref{rationalnormalscroll}).
\item[(iii)] $X=\P^2$, $L=M=\O_{\P^2}(1)$.
\end{enumerate}
\end{corollary}

\begin{proof} 
The inequality follows from Theorem \ref{teoremon}(i) using that
$\dim J(F)\le\dim\tilde J(F)$. In case of equality, we have to
remove from Theorem \ref{megateoremon} the cases in which $\dim
J(F)<\dim\tilde J(F)$. Observe that the case $t=s+1$ in Example
\ref{superficieminimal} (i.e. when $\dim J(F)=\dim\tilde J(F)=1$)
becomes the case $n=1$ in Example
\ref{schwarzenberger}, so that we do not need to consider it.
\end{proof}

We also have this improvement of Re's results in the case of line
bundles:

\begin{corollary} \label{re}
Let $L,M$ be two globally generated line bundles on an irreducible
variety $X$, and assume that $L\otimes M$ is ample. Then
$h^0(L\otimes M)\ge h^0(L)+h^0(M)+\dim(X)-2$, with equality if and
only if there is a triplet $(X',L',M')$ as in Examples
\ref{schwarzenberger}, \ref{superficieminimal},
\ref{rationalnormalscroll} or \ref{Veronese} such that there exists
a finite map $f:X\to X'$ satisfying $L=f^*L'$ and $M=f^*M'$.
\end{corollary}

\begin{proof} 
Let $F$ be the Schwarzenberger bundle of the triplet $(X,L,M)$.
Then $\tilde J(F)$ is the image of $X$ via $L\otimes M$. Since
$L\otimes M$ is ample and globally generated, it follows
$\dim(\tilde J(F))=\dim(X)$. Thus the wanted inequality is just
Theorems \ref{teoremon}(i). In case we have equality, we know by
Theorem \ref{megateoremon} that $F$ is the Schwarzenberger bundle
of a triplet $(X',L',M')$ as in Examples \ref{schwarzenberger},
\ref{superficieminimal}, \ref{rationalnormalscroll} or \ref{Veronese}.
Moreover, the proof gives that $X'$ is $\tilde J(F)$, i.e. the
image of $X$ via the map $f$ defined by $L\otimes M$. Also, since
the composition
$X\mapright{f}X'\mapright{\pi_1}\P(H^0(L))$ is the map defined by $L$,
it follows $L=f^*\pi_1^*\O_{\P(H^0(L))}(1)=f^*L'$, and similarly we
obtain $M=f^*M'$.
\end{proof}

\begin{remark} \label{novaleV} {\rm
It could seem a priori that it is possible to obtain Theorem
\ref{megateoremon} as a Corollary of the corresponding result of
\cite{V1} for Steiner bundles of rank $n$. In fact, we can always take a
general quotient $T\to T_1$ of dimension
$n+s$ and, if $K$ is its kernel, we get a commutative diagram
$$\begin{matrix}
&&&&0&&0\cr
&&&&\downarrow&&\downarrow\cr
&&&&K\otimes\O_{\P^n}(-1)&=&K\otimes\O_{\P^n}(-1)\cr
&&&&\downarrow&&\downarrow\cr
0&\to&S\otimes\O_{\P^n}(-1)&\to&T\otimes\O_{\P^n}&\to&F&\to&0\cr
&&||&&\downarrow&&\downarrow\cr
0&\to&S\otimes\O_{\P^n}(-1)&\to&T_1\otimes\O_{\P^n}&\to&F_1&\to&0\cr
&&&&\downarrow&&\downarrow\cr
&&&&0&&0
\end{matrix}$$
in which now $F_1$ is a Steiner bundle of rank $n$. From this
diagram, it is not difficult to see that $\tilde J(F_1)$ is the
intersection of $\tilde J(F)$ with $\P(T_1)$. Since $\P(T_1)$ has
codimension $t-n-s$ in $\P(T)$, it follows that $\dim\tilde
J(F_1)\ge\dim\tilde J(F)-t+n+s$. Since the dimension of $\tilde
J(F_1)$ is at most one (by Theorem \ref{teoremon}(i), which is in this case
the result of \cite{V1}), it follows that $\tilde J(F)$ has dimension at
most $t-n-s+1$. Moreover, if equality holds, we can apply the known
result for $F_1$ and get that $\tilde J(F_1)$ is a rational normal
curve, so that $\tilde J(F)$ has only one component of maximal
dimension, which is a variety of minimal degree in $\P(T)$.
However, such a proof does not exclude the possibility that
$\tilde J(F)$ (or
$J(F)$) has other components of smaller dimension, while our
proof shows the irreducibility of $\tilde J(F)$. Hence our proof
actually provides a positive answer to  Question \ref{otrapregunta}
for the Examples \ref{schwarzenberger}, \ref{superficieminimal},
\ref{rationalnormalscroll} and \ref{Veronese}.
}\end{remark}

\begin{remark} \label{nadamas} {\rm
The proof of Theorem \ref{megateoremon} gives an idea of the
difficulty of proving a similar result for arbitrary $a,b$.
Independently of the fact that we were not able to find a
reasonable bound for the dimension of $J_{a,b}(F)$, the main
obstacle to prove something analogous to Theorem
\ref{megateoremon} is that we do not have a first induction step
to apply an iteration using Proposition \ref{transform}. Indeed,
the minimal value of $s$ would be $s=a+1$ (see Lemma
\ref{tecnico}), but as observed in Remark \ref{pocosjumping}, a
result like Theorem \ref{teoremon}(iv) cannot hold because, for
general values of $a,b$, one expects
$\tilde J_{a,b}(F)$ to be empty, even for $s=a+1$. The same
problem remains when trying to apply the iteration process
explained in Remark \ref{simetria}, since the first step should be a
Steiner bundle on $\P^{b+1}$, for which we also expect $\tilde
J_{a,b}(F)$ to be empty for general values of $a,b$.

On the other hand, it would also be nice to generalize Theorem
\ref{megateoremon} to arbitrary $a,b$ in order to generalize the
improvement of Re's results given in Corollary \ref{re} to
arbitrary rank. Since our proof for $a=b=1$ is closely related to
the classification of varieties of minimal degree in the projective
space, a generalization to arbitrary $a,b$ is likely to depend on
a good theory of varieties of minimal degree in Grassmannians (see
\cite{Si} for a first natural approach). 
}\end{remark}

\begin{remark} \label{Soares} {\rm
In \cite{So}, Soares gave a natural definition of Steiner bundle on any
projective variety. It would be nice to have also the notion of
Schwarzenberger bundle in her general context. For example, to get
a natural definition on Grassmannians, one could take a triplet
$(X,L,M)$ and fix an integer $r$ such that, for each
$r$-dimensional subspace $V\subset H^0(M)$ the natural map
$H^0(L)\otimes V\to H^0(L\otimes M)$ is injective. Let us consider
$G=G(r,H^0(M))$, the Grassmann variety of linear subspaces of
dimension
$r$ in $H^0(M)$, and let $\U$ be the rank $r$ universal
subbundle of $G$. Then there is an exact sequence of
vector bundles on $G$:
$$0\to H^0(L)\otimes\U\to H^0(L\otimes M)\otimes\O_G\to
F\to0$$
defining $F$ as a cokernel. This is a Steiner bundle on
$G$ in the sense of \cite{So}, so that it seems natural to define
Schwarzenberger bundles on $G$ as the bundles obtained in this way.
Of course, when $r=1$ we recover our definition of Schwarzenberger
bundle on the projective space. 
}\end{remark}


\bigskip

\centerline{Departamento de \'Algebra}
\centerline{Facultad de Ciencias Matem\'aticas} 
\centerline{Universidad Complutense de Madrid}
\centerline{28040 Madrid, Spain}
\centerline{arrondo@mat.ucm.es}

\end{document}